\documentclass{amsart}

\usepackage{amssymb}
\usepackage{mathrsfs}
\usepackage{amscd}
\usepackage[normalem]{ulem}
\usepackage[pagebackref]{hyperref}
\usepackage{graphicx}
\usepackage[curve]{xypic}

\let\cal\mathcal
\renewcommand{\setminus}{\smallsetminus}
\newcommand\R{{\mathbb R}}
\newcommand\C{{\mathbb C}}
\newcommand{\grassmann}{\mathbf{G}}
\newcommand{\norm}[1]{\lVert #1 \rVert}

\newtheorem{theorem}{Theorem}[section]
\newtheorem{proposition}[theorem]{Proposition}

\newtheorem{lemma}[theorem]{Lemma}

\theoremstyle{definition}
\newtheorem{example}[theorem]{Example}
\newtheorem{remark}[theorem]{Remark}
\newtheorem{definition}[theorem]{Definition}

\newcommand{\proj}{\operatorname{proj}}
\newcommand{\Res}{\operatorname{Res}}
\newcommand{\mult}{\operatorname{mult}}

\begin{document}
\title[Lipschitz normal embedding among superisolated singularities]{Lipschitz normal embedding among superisolated singularities}

\author{Filip Misev}
\address{Max Planck Institute for Mathematics, Vivatsgasse~7, 53111~Bonn, Germany} \email{fmisev@mpim-bonn.mpg.de}

\author{Anne Pichon}
\address{Aix Marseille Universit\'e, CNRS, Centrale Marseille, I2M, UMR 7373, 13453 Marseille, FRANCE}
\email{anne.pichon@univ-amu.fr}


\begin{abstract}
Any germ of a complex analytic space is equipped with two natural metrics: the outer metric induced by the hermitian metric of the ambient space and the inner metric, which is the associated riemannian metric on the germ. A complex analytic germ is said Lipschitz normally embedded (LNE) if its outer and inner metrics are bilipschitz equivalent. LNE seems to be fairly rare among surface singularities; the only known LNE surface germs outside the trivial case (straight cones) are the minimal singularities. In this paper, we show that a superisolated hypersurface singularity is LNE if and only if its projectivized tangent cone has only ordinary singularities. This provides an infinite family of LNE singularities which is radically different from the class of minimal singularities.
\end{abstract}

\maketitle

\section{Introduction}

If $(X,0)$ is a germ of a
complex variety, then any embedding $\phi\colon(X,0)\hookrightarrow
(\C^n,0)$ determines two metrics on $(X,0)$: the outer metric
$$d_o({x,y}):=\norm{{\phi(x)-\phi(y)}}\qquad\text{(i.e., distance in $\C^n$)}$$
and the inner metric
$$d_i(x,y):=\inf\{\mathop{\rm length}(\phi\circ\gamma):
\gamma\ \text{is a rectifiable path in}\ X\ \text{from}\ x\ \text{to}\ y\},$$
using the riemannian metric on $X \setminus \{0\}$ induced by the hermitian metric of $\C^n$. For all $x,y \in X, \ \norm{x-y} \leq d_i(x,y)$.

\begin{definition} A germ of a complex normal variety $(X,0)$ is \emph{Lipschitz normally embedded} (LNE) if the identity map of $(X,0)$ is a bilipschitz homeomorphism between inner and outer metrics, i.e., there exists a neighborhood $U$ of $0$ in $X$ and a constant $K\geq 1$ such that for all $x,y \in U$
$$d_i(x,y) \leq K d_o(x,y).$$
\end{definition}
 
Lipschitz Normal Embedding (LNE) is a very active research area with many recent results, e.g., by Birbrair, Fernandes, Kerner,
 Mendes, Neumann, Nu\~no-Ballesteros,  Pedersen, Pichon, Ruas, Sampaio (\cite{BMN}, \cite{FS}, \cite{KMR}, \cite{MR}, \cite{NPP2}), including a characterization of LNE  for semialgebraic sets (\cite{BM1}) and a characterization of LNE for complex surfaces (\cite{NPP1}).

If $(X,0)$ is a curve germ then it is in fact bilipschitz equivalent to the metric cone over its link with respect to the inner metric, while the data of  its Lipschitz outer geometry is equivalent to that of the embedded topology of a generic plane projection (see \cite{PT}, \cite{F}, \cite{NP}). Therefore, an irreducible complex curve is LNE if and only if it is smooth. This is no longer true in higher dimension.

For example, if $f_d \in \C[x,y,z]$ is a homogeneous polynomial of degree $d\geq 2$ such that the associated projective curve $\{f_d=0\} \subset \mathbb{P}^2$ is smooth,  then it is easy to prove that the hypersurface in $\C^3$ with isolated singularity defined by the equation $f_d(x,y,z)=0$, i.e. the straight cone over the projective curve $f_d=0$,  is LNE.  A~natural question is then to characterize all the isolated surface singularities which are LNE.

It turns out that it is difficult to find examples of complex isolated singularities with dimension $\geq 2$ which are LNE but which are not  conical. In \cite{NPP2}, it is shown that such examples exist: every minimal surface singularity is LNE. This was the first known infinite class of LNE isolated surface singularities which are not  conical.

In the present paper, we prove that the class of superisolated singulaties in $(\C^3,0)$ contains an infinite amount of non metrically conical isolated singularities which are LNE and we describe them precisely (Theorem \ref{th:main}).  This class of examples is very different from the class of minimal  singularities. Indeed,  in general, a minimal singularity has (by definition) a large embedding dimension. Among them, the only  hypersurfaces in $(\C^3,0)$  are the singularities $A_k, k \geq 1$.

\begin{definition}
A hypersurface singularity $(X,0) \subset (\C^3,0)$ is {\it superisolated} if it is given by an equation 
$$ f_{d} (x,y,z)+f_{d+1} (x,y,z)+f_{d+2}(x,y,z)+\ldots=0,$$
where $d \geq 2$, for all $k\geq d$, $f_k$ is a homogeneous polynomial of degree $k$ and the projective curve $\{f_{d+1}=0 \} \subset \mathbb{P}^2$ does not contain any singular point of the projectivized tangent cone   $C_0X= \{[x:y:z] \in \mathbb{P}^2 \colon f_d (x,y,z)=0\}$. In particular, the curve  $C_0X$ is reduced. 
\end{definition}  

The embedded topology of a superisolated singularity is completely determined by the combinatorial type of the projective curve $f_d=0$, i.e. by the topology of a small tubular neighbourhood of $f_d=0$ in $\mathbb{P}^2$. Then, when considering only embedded topology, we can assume that the equation is  simply $f_d(x,y,z) + l^{d+1} =0$ where $l(x,y,z) = ax+by+cz$ is a generic linear form.

In this paper, we consider the outer Lipschitz geometry  of superisolated surface singularities with equations of the form:
$$ f_{d} (x,y,z)+f_{d+1} (x,y,z)=0, $$
i.e., $f_{d+k} =0$ for all $k \geq 2$. 

\begin{definition} \label{def:ordinary}
A reduced complex curve singularity $(C,p)$ is \emph{ordinary} if it consists of $r$ smooth curve germs having $r$ distinct tangent lines at $p$.
\end{definition} 

\begin{theorem} \label{th:main} A superisolated singularity  $(X,0) \subset (\C^3,0)$ with equation $f_{d} +f_{d+1} =0$  is LNE if and only if  its projectivized tangent cone $C_0X$ has only ordinary singularities.
\end{theorem}
  
The proof of Theorem \ref{th:main} will use the Test Curve Criterion \cite[Theorem 3.8]{NPP1} which is a characterization of LNE for normal complex surface singularities. We recall its statement in Section \ref{sec:characterization LNE}.

Notice that  superisolated singularities were already used   in the context of Lipschitz geometry of singularities to provide examples of surface singularities in $(\C^3,0)$ having same outer Lipschitz geometry but different embedded topological type (see~\cite{NP2}).

\vskip0.1cm\noindent{\bf Acknowledgments.}  
Anne Pichon was supported by the ANR project LISA  17-CE40--0023-01. Filip Misev was supported by the Max Planck Institute for Mathematics in Bonn.

\section{Examples}
In view of Theorem~\ref{th:main}, it becomes straightforward to explicitly construct super\-isolated surface singularities in $\C^3$ that are LNE. It suffices to consider hypersurfaces of the form $f_d+f_{d+1}=0$,  where $f_d$ is carefully chosen to have only ordinary singularities. We illustrate this below with  two infinite families of LNE superisolated singularities and one example of a superisolated singularity which fails to be LNE. We choose here $f_{d+1}$ to be the power of a linear form, which allows us to check easily that the corresponding singularities are superisolated.

\begin{example}[$C_0X$ smooth curve in $\mathbb{P}^2$]
Let $f_d\in\C[x,y,z]$ be a homogeneous polynomial of degree $d$ such that the projective curve $\{ f_d=0 \}\subset\mathbb{P}^2$ is smooth. Let $l=ax+by+cz$ be any linear form. Then the singularity $(X,0)\subset (\C^3,0)$ defined by the equation
$$ f_d + l^{d+1} = 0 $$
is LNE. Indeed, since the projectivized tangent cone $C_0X=\{f_d=0\}$ does not have any singularities, the condition of Theorem~\ref{th:main} on the singularities of $C_0X$ is vacuously satisfied.
\end{example}

\begin{example}[$C_0X$ generic intersection of smooth curves]
Let $f_{d_1},\ldots,f_{d_r}\in\C[x,y,z]$ be homogeneous polynomials of respective degrees $d_1,\ldots, d_r$ such that the $r$ associated projective curves are smooth, pairwise transverse and without triple intersection points. Let $d=d_1+\ldots+d_r$. Take  any linear form $l$ whose zero locus avoids the singular points of  the curve $\{f_{d_1}\cdot\ldots\cdot f_{d_r} =0\}$. Then the superisolated singularity $(X,0)$ defined by the equation
$$ f_{d_1}\cdot\ldots\cdot f_{d_r} + l^{d+1}=0 $$
is LNE, since $C_0X=\{ f_{d_1}\cdot\ldots\cdot f_{d_r} = 0 \}$ has only ordinary singularities.
\end{example}

\begin{example}[$C_0X$ cuspidal curve]
For every generic    linear form $l$, the hypersurface $(X,0)\subset(\C^3,0)$ defined by the equation
$$zx^2+y^3 + l^4 = 0$$
is superisolated, but not LNE: the projectivized tangent cone $C_0X\subset\mathbb{P}^2$ has a singularity at $[0:0:1]$ which is a cusp, hence not ordinary.
\end{example}

\section{The Test Curve Criterion for LNE of a surface singularity}\label{sec:characterization LNE}

Before stating the Test Curve Criterion  \cite[Theorem 3.8]{NPP1}, we need to introduce some material.

\subsection{Generic projections}

Let $\cal D$ be a $(n-2)$-plane in $\C^n$ and let $\ell_{\cal D} \colon \C^n \to \C^2$ be the linear projection with kernel $\cal D$. Suppose $(C,0)\subset (\C^n,0)$ is a complex curve germ.  There exists an open dense subset $\Omega_C \subset \grassmann(n-2,\C^n)$  such that for $\cal D \in \Omega_C$, $\cal D$ contains no limit of bisecant lines to the curve $C$ (\cite{teissier}). The  projection $\ell_{\cal D}$ is  said  to be \emph{generic for $C$} if  $\cal D \in \Omega_C$.

Let now $(X,0)\subset (\C^n,0)$ be a normal surface singularity. We restrict ourselves to those $\cal D$ in the Grassmannian $\grassmann(n-2,\C^n)$ such that the restriction $\ell_{\cal D}{\mid_{(X,0)}}\colon(X,0)\to(\C^2,0)$ is finite.
The \emph{polar curve} $\Pi_{\cal D}$ of $(X,0)$ for the direction $\cal D$ is the closure in $(X,0)$ of the singular locus of the restriction of $\ell_{\cal D} $ to $X \setminus \{0\}$. The \emph{discriminant curve} $\Delta_{\cal D} \subset (\C^2,0)$ is the image $\ell_{\cal D}(\Pi_{\cal D})$ of the polar curve $\Pi_{\cal D}$.

\begin{proposition}[{\cite[Lemme-cl\'e V 1.2.2]{teissier}}]\label{prop:generic} An open dense subset $\Omega \subset \grassmann(n-2,\C^n)$ exists such that: 
\begin{enumerate}
\item \label{cond:generic1} the family of  curve germs  $(\Pi_{\cal D})_{\cal D \in \Omega}$
is equisingular in terms of strong simultaneous resolution;
\item \label{cond:generic2} the discriminant curves  $ \Delta_{\cal D}=\ell_{\cal D}(\Pi_{\cal D})$, ${\cal D \in \Omega}$, form an equisingular family of reduced plane curves;
\item \label{cond:generic3}  for each $\cal D$, the projection $\ell_{\cal D}$ is generic for its polar curve $\Pi_{\cal D}$. 
\end{enumerate}
\end{proposition}

\begin{definition}  \label{def:generic linear projection} 
The projection $\ell_{\cal D} \colon \C^n \to \C^2$ is \emph{generic} for $(X,0)$ if $\cal D \in \Omega$.
\end{definition}

\subsection{Test curves}

Let  $\ell \colon (X,0) \to (\C^2,0)$ be a generic projection, let $\Pi $ be its polar curve and let $\Delta = \ell(\Pi)$ be its discriminant curve. Denote by $\rho'_{\ell} \colon Y_{\ell} \to \C^2$ the minimal composition of blow-ups of points starting with the blow-up of the origin which resolves the base points of the family of projections of generic polar curves  $(\ell(\Pi_{\cal D}))_{\cal D \in \Omega}$.

\begin{definition}  \label{def:Delta-curve}  
A \emph{$\Delta$-curve}  is  an irreducible component of the exceptional curve {$(\rho'_{\ell})^{-1}(0)$} intersecting the strict transform of $\Delta$.  
\end{definition}
Let us blow-up all the intersection points between two $\Delta$-curves. We denote by $\sigma \colon Z_{\ell} \to Y_{\ell}$ and $\rho_{\ell}=\rho'_\ell\circ \sigma\colon Z_\ell\to\C^2$ the resulting morphisms (if no $\Delta$-curves intersect, $\rho_ {\ell} =  {\rho'_{\ell}}$).   The resolution graph of $\rho_{\ell}$ does not depend on $\ell$. We denote it by  $T$ for the rest of the paper.  
\begin{definition}  \label{def:Delta-node}  
A  {\it $\Delta$-node} of $T$ is a vertex $(j)$ of $T$ which represents a $\Delta$-curve. 
\end{definition}
   
Let $E \subset Y$ be a complex curve in a complex surface $Y$ and let $E_1,\ldots,E_n$ be the irreducible components of $E$. We say {\it curvette} of $E_i$ for any smooth curve germ $(\beta,p)$ in $Y$, where $p$ is a point of $E_i$ which is a smooth point of $Y$ and $E$ {and} such that  $\beta$ and $E_i$ intersect transversely. 

If $G$ is a graph, we will denote by $V(G)$ its set of vertices.

\begin{definition} \label{def:test curve}
Let $T'$ be the subtree of $T$ defined as the union of all the simple paths in $T$ joining the root vertex to $\Delta$-nodes (so the complement $T \setminus T'$ consists of strings of valency $2$ vertices ended by a valency $1$ vertex).

For $(i) \in V(T)$, let $C_i$  be the irreducible component of $\rho_{\ell}^{-1}(0)$ represented by $(i)$, so we have $\rho_{\ell}^{-1}(0) = \bigcup_{(i) \in V(T)} C_i$.
Let $(i) \in V({T})$. We call \emph{test curve at $(i)$} (of~$\ell$)  any complex curve germ $(\gamma,0) \subset (\C^2,0)$ such that 
\begin{enumerate}
\item $(i) \in V(T')$;
\item the strict transform $\gamma^*$   by  $\rho_\ell$ is a curvette of $C_i$  intersecting $C_i$ at a smooth point of $\rho_{\ell}^{-1}(0)$;
\item $\gamma^* \cap \Delta^* = \emptyset$. 
\end{enumerate}
\end{definition}

\subsection{Principal components}

Let us first recall the definition of Nash modification. 

Let $\lambda \colon X\setminus\{0\} \to \grassmann(2,\C^n)$ be the map which sends $x \in X\setminus\{0\}$ to the tangent plane $T_xX$. The closure $N X$ of the graph of $\lambda$ in $X \times \grassmann(2,\C^n)$ is a reduced analytic surface. By definition, the \emph{Nash modification} of $(X,0)$ is the induced morphism $\mathscr N  \colon NX \to X$. A morphism $f \colon Y \to X$ factors through Nash modification if and only if it has no base points for the family of polar curves ({\cite[Section 2]{GS}, \cite[Part~III, Theorem~1.2]{S}).

\begin{definition} \label{def:node G resolution} Let $\pi_0 \colon X_0 \to X$ be the minimal good  resolution of $X$ which factors through both the Nash modification $\mathscr N \colon NX \to X$ of $(X,0)$ and  the blow-up of the maximal ideal  of $(X,0)$ and let $G_0$ be its resolution graph. For each vertex $(v)$ of $G_0$  we denote by $E_v$ the corresponding  irreducible component of $\pi_0^{-1}(0)$. A~vertex $(v)$ of $G_0$ such that $E_v$ is an irreducible component of the blow-up of the maximal ideal (resp.\ an exceptional curve of the Nash transform)  is called an {\it $\cal L$-node} (resp.\ a {\it $\cal P$-node}) of $G_0$.
\end{definition}

\begin{definition} \label{def:G'} Consider the graph $G'_0$ of $G_0$ defined as the union of all simple paths in  $G_0$ connecting  pairs of  vertices among $\cal L$- and $\cal P$-nodes. Let $\ell \colon (X,0) \to (\C^2,0)$ be a generic projection. Let  $\gamma$ be a test curve for $\ell$. A component $\widehat{\gamma}$ of $\ell^{-1}(\gamma)$ is called {\it principal} if its  strict transform by $\pi_0$ either is a curvette of a component $E_v$ with $v \in V(G'_0)$ or intersects $\pi_0^{-1}(0)$ at an intersection between two exceptional curves  $E_v$ and $E_{v'}$ such that both  $v$ and $v'$ are in $V(G'_0)$.
\end{definition}

\subsection{Inner and outer contact exponents}  We now define the outer and inner contacts between two complex curves on a complex surface germ. We use the ``big-Theta" asymptotic notation of Bachmann-Landau:  given two function germs $f,g\colon ([0,\infty),0)\to ([0,\infty),0)$ we say $f$ is \emph{big-Theta} of $g$ and we write   $f(t) = \Theta (g(t))$ if there exist real numbers $\eta>0$ and $K \ge 1$ such that for all $t$ with $f(t)\le\eta$,  $\frac{1}{K }g(t) \leq f(t) \leq K g(t)$.

Let $\mathbb{S}^{2n-1}_{\epsilon} = \{ x \in \C^n \colon \norm{x}_{\C^n} = \epsilon\}$. Let $(\gamma_1,0)$ and $(\gamma_2,0)$ be two germs of complex curves inside $(X,0)$. Let $q_{out}=q _{out}(\gamma_1, \gamma_2)$ and $q_{inn}=q_{inn}(\gamma_1, \gamma_2)$ be the two rational numbers $\geq 1$ defined by 
$$ d_o(\gamma_1 \cap \mathbb{S}^{2n-1}_{\epsilon}, \gamma_2 \cap \mathbb{S}^{2n-1}_{\epsilon}) =  \Theta(\epsilon^{q_{out}}),$$
$$ d_i(\gamma_1 \cap \mathbb{S}^{2n-1}_{\epsilon}, \gamma_2 \cap \mathbb{S}^{2n-1}_{\epsilon}) =  \Theta(\epsilon^{q_{inn}}),$$
where $d_{i}$ means inner distance in $(X,0)$ as before. 

\begin{definition} We call $q _{out}(\gamma_1, \gamma_2)$ (resp.\ $q _{inn}(\gamma_1, \gamma_2)$) the {\it outer  contact exponent} (resp.\ {\it  the inner  contact exponent}) between $\gamma_1$ and $\gamma_2$.
\end{definition}


We are now ready to state the Test Curve Criterion \cite[Theorem~3.8]{NPP1}:

\begin{theorem}[{Test curve criterion for LNE of a complex surface}] \label{thm:characterization LNE} A normal surface germ  $(X,0)$  is LNE if and only if the following conditions are satisfied for {all  generic projections} $\ell \colon (X,0) \to (\C^2,0)$ and all  test curves $(\gamma,0) \subset (\C^2,0)$: 
\begin{itemize}
\item[$(1^*)$] \label{iso} for every  principal component $\widehat{\gamma}$ of $\ell^{-1}(\gamma)$, $\mult(\widehat{\gamma})=\mult({\gamma})$ where $\mult$ means multiplicity at $0$; 
\item[$(2^*)$] \label{vertical} for every pair $(\gamma_1,\gamma_2)$ of distinct principal components of $\ell^{-1}(\gamma)$, $q_{inn}(\gamma_1,\gamma_2) = q_{out}(\gamma_1,\gamma_2)$.
\end{itemize}
\end{theorem}

\section{The trees $T$ and $G_0$ of an superisolated singularity with ordinary projectivized tangent cone} \label{sec:T and G0}

In this section, we prove the following proposition which will enable one to describe the graphs $T$ and $G_0$ for a superisolated singularity whose projectivized tangent cone has only ordinary singularities.
 
\begin{proposition} \label{prop:Delta} Let $(X,0) \subset (\C^3,0)$ be a superisolated singularity and let $e \colon X^* \to X$ be the blow-up of the maximal ideal of $X$, so the projectivized tangent cone $C_0X$ is  the exceptional curve $ e^{-1}(0)$. Let $p_i, i=1, \ldots, r $ be the singular points of $C_0X$. 
 
Let $\cal D \in \Omega$, i.e., $\ell_{\cal D}\colon  (X,0) \to (\C^2,0)$ is a generic projection, and let us decompose the polar curve $\Pi_{\cal D}$ as the union
$$\Pi_{\cal D}= \Pi_{\cal D,0} \cup \Pi_{\cal D,1} \cup \ldots \cup \Pi_{\cal D,r},$$
where the strict transform $\Pi_{\cal D,0}^*$ of $\Pi_{\cal D,0}$ by $e$ intersects $C_0X$ at smooth points and where for each $i=1,\ldots,r$, $ \Pi_{\cal D,i}^*$ consists of the components of $\Pi_{\cal D}^*$ passing through~$p_i$. So in particular, for every $i \neq j$, $\Pi_{\cal D,i}$ and $\Pi_{\cal D,j}$ are transversal.

\begin{enumerate}
\item For each $i=1,\ldots, r$, $\Pi_{\cal D,i} \neq \emptyset$  and the basepoints of the family of generic polar curves $\big( \Pi_{\cal D'} \big)_{\cal D' \in \Omega}$ are exactly the points  $p_i$, $i=1,\ldots, r$.
\item Assume that the projectivized tangent cone of $(X,0)$ has only ordinary singularities, i.e., for each $i=1,\ldots, r$, the germ $(C_0X,p_i)$ consists of $k_i$ smooth transversal curves. Then
\begin{enumerate}
\item The germ $(\Pi_{\cal D,i},0)$  has $k_i-1$ components,  its strict transform by $e$ consists of  smooth transversal curves passing through $p_i$ and the basepoint $p_i$ is resolved by just one blow-up.
\item Let $\ell \colon (X,0) \to (\C^2,0)$   be a generic projection and let  $\cal D \in \Omega$. For each $i=1,\ldots, r$, the $k_i-1$ components of the curve $\ell(\Pi_{\cal D,i})$ are all $(k_i+1,k_i)$-cusps and they have pairwise contact $\frac{k_i+1}{k_i}$.   Moreover, for every generic pair $\cal D, \cal D' \in \Omega$, any two components of $\ell(\Pi_{\cal D,i}) \cup \ell(\Pi_{\cal D',i})$ have  contact $\frac{k_i+1}{k_i}$.   
\end{enumerate}
\end{enumerate}
\end{proposition}

As a consequence of Proposition \ref{prop:Delta}, we obtain the following description of the resolution tree $T$ and of the resolution graph $G_0$ introduced in Section \ref{sec:characterization LNE} for all superisolated singularities whose projectivized tangent cones  have only ordinary singularities.
 
Denote by $T_i$ the resolution tree of the $(k_i,k_i+1)$-cusp  decorated with $k_i-1$ arrows corresponding to the components of $\ell(\Pi_{\cal D,i})$ (see Figure~\ref{fig:Ti}). The weight decorating each vertex is the self-intersection of the corresponding exceptional curve. The root vertex, i.e., the vertex corresponding to the first blow-up of $0_{\C^2}$ is the circled one.  The negative number weighting each vertex is the self-intersection of the corresponding exceptional curve. 

\begin{figure}[ht]
\includegraphics{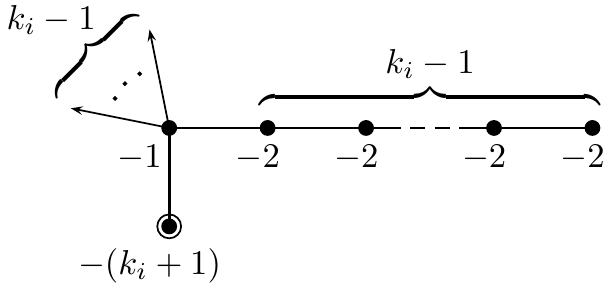}
\caption{The tree $T_i$.}
\label{fig:Ti}
\end{figure}

By Proposition \ref{prop:Delta}, the tree $\cal T$ of the minimal sequence $\rho'_{\ell}$ of blow-ups which resolves the family of projected generic polar curves $\ell(\Pi_{\cal D})$  is a bouquet of $r$ trees $T_i$, $i=1,\ldots,r$ attached by their root vertex (Figure~\ref{fig:T}). Moreover, any $\Delta$-node of~$\cal T$ which is not the root vertex is the vertex of a $T_i$ with valency $k_i+1$ and each such vertex is a $\Delta$-node.

\begin{figure}[ht]
\includegraphics{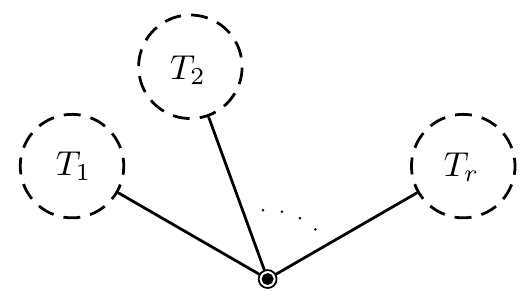}
\caption{The tree $\cal{T}$ is the bouquet of the $r$ trees $T_1,\ldots,T_r$, obtained by identifying their root vertices.}
\label{fig:T}
\end{figure}

By definition, $T$ is obtained from $\cal T$ by blowing-up each edge joining two $\Delta$-nodes, creating separating nodes. In our situation,  there are two cases.
\vskip1em

\noindent {\it Case 1.} Either $C_0X$ is a line arrangement, i.e., each of its components has degree~$1$, then the root vertex is not a $\Delta$-node and  there are no  adjacent  $\Delta$-nodes in~$\cal T$. We then have $\rho_{\ell} = \rho'_{\ell}$, $\Pi_{\cal D,0} = \emptyset$ and  $T=\cal T$. We obtain the  tree $T$  of Figure \ref{fig:TCase1}. 
 
\begin{figure}[ht]
\includegraphics{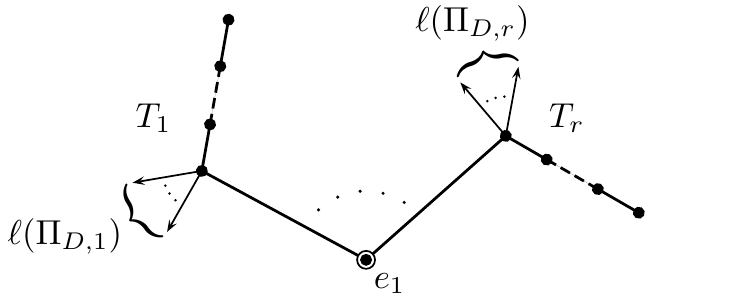}
\caption{The tree $T$ in Case 1. The multiplicity at the root vertex is $e_1=-1-\sum_{i=1}^r k_i$.}
\label{fig:TCase1}
\end{figure}

\noindent {\it Case 2.} Otherwise, i.e., when at least one of the components of $C_0X$ has degree $\geq 2$,  then the root vertex of $\cal T$ is a $\Delta$-node and every adjacent edge has the $\Delta$-node of a $T_i$ at the other extremity. The tree $T$ is then obtained  from $\cal T$ by blowing-up each edge adjacent to the root vertex, creating $r$ new vertices.  Let us denote $T'_i$ the tree resulting from $T_i$ by this blowing-up. The tree $T$ is then a bouquet of trees $T'_i$  attached by their root vertex. We obtain the tree $T$ of Figure~\ref{fig:TCase2}.

\begin{figure}[ht]
\includegraphics{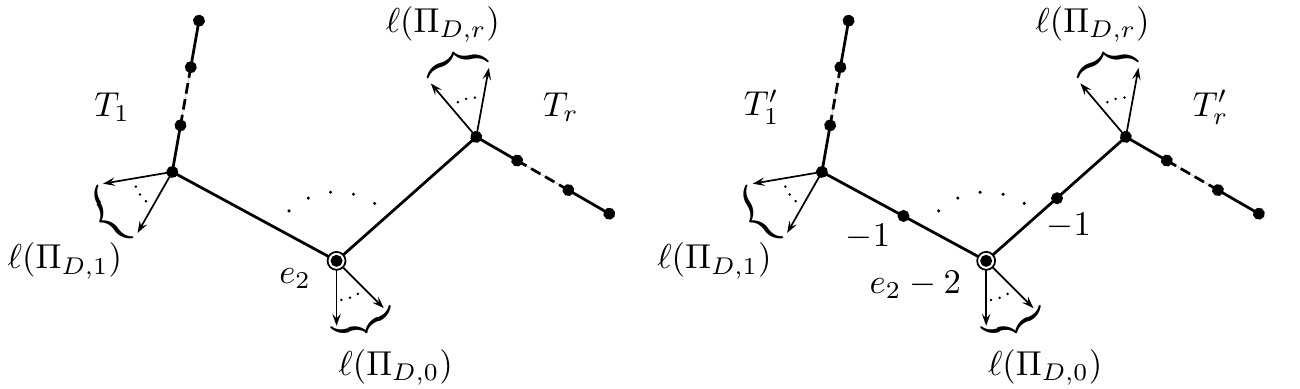}
\caption{The trees $\cal T$ (left) and $T$ (right) in Case 2. Here, $e_2=e_1-r$.}
\label{fig:TCase2}
\end{figure}

The exceptional divisor of $\pi_0^{-1}(0)$ is obtained by blowing-up each singular point of $C_0X$ and  the created $\mathbb{P}^1$ are the $\cal P$-curves of  $\pi_0^{-1}(0)$, i.e., the curves corresponding to the components of the Nash transform of $(X,0)$ (or equivalently, those intersecting the strict transform of the polar curve $\Pi$ of a generic projection). This describes~$G_0$.  

\begin{remark}
In the case of a superisolated singularity, we   have $G'_0 = G_0$ since every vertex of $G_0$ is  either an $\cal L$-node or a $\cal P$-node.
\end{remark}

\begin{example} \label{ex:1} Consider the superisolated singularity $(X,0)$ with equation
$$ xy(x+y)(x-y) + z^5 = 0. $$
Its projectivized tangent cone $C_0X=\{ xy(x+y)(x-y)=0 \}\subset\mathbb{P}^2$ consists of four lines which intersect transversely at the point $[0:0:1]$. In particular, the unique singularity of $C_0X$ is ordinary.  The circled vertices in $G_0$ represent components of $C_0X$.  The arrows in $T$ represent the components of  the discriminant curve $\Delta$ and  arrows in $G_0$ represent the components of  the polar curve $\Pi$. We do not write all intersections numbers (Euler classes) since we do not use them in the paper, but they can be easily computed.

\begin{figure}[ht]
\includegraphics{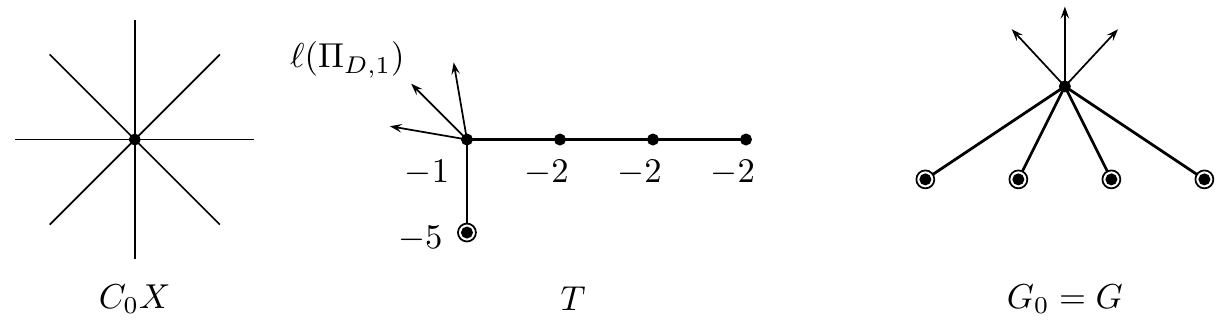}
\caption{Example~\ref{ex:1}: Lines intersecting at a single point.}
\label{fig:Ex1}
\end{figure}
\end{example}

\begin{example} \label{ex:2}
Let $(X,0)$ be given by the equation
$$ xy(x^2+y^2+z^2) + z^5 = 0. $$
Here, the root vertex of $T$ and all adjacent vertices are $\Delta$-nodes, so we have to blow-up the intersection points between the corresponding $\Delta$-curves, creating four vertices which are weighted by the Euler class $-1$ on the tree below. 
\begin{figure}[ht]
\includegraphics{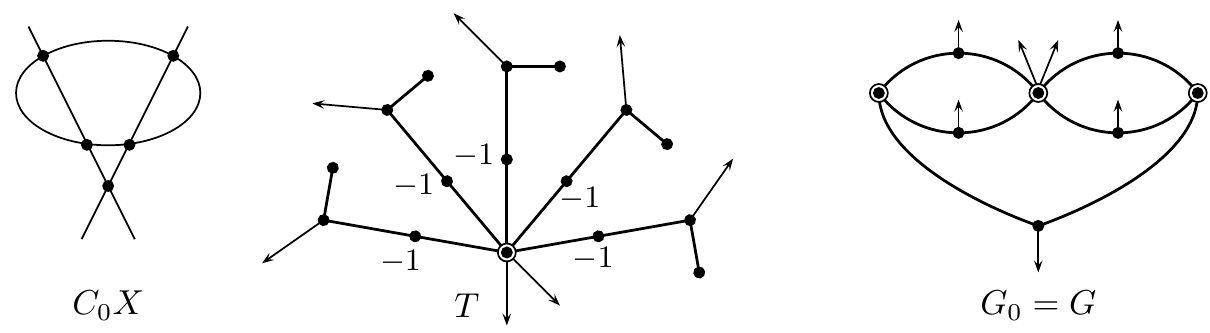}
\caption{Example~\ref{ex:2}: A pair of lines intersecting an oval.}
\label{fig:Ex2}
\end{figure}
\end{example}

\begin{example} \label{ex:3}
Consider the singularity $(X,0)\subset (\C^3,0)$, for which $C_0X$ consists of two transverse ovals: 
$$ (x^2+2y^2)(2x^2+y^2) + z^5 = 0.$$ The graphs $T$ and $G_0$ are represented on Figure \ref{fig:Ex3}.

\begin{figure}[ht]
\includegraphics{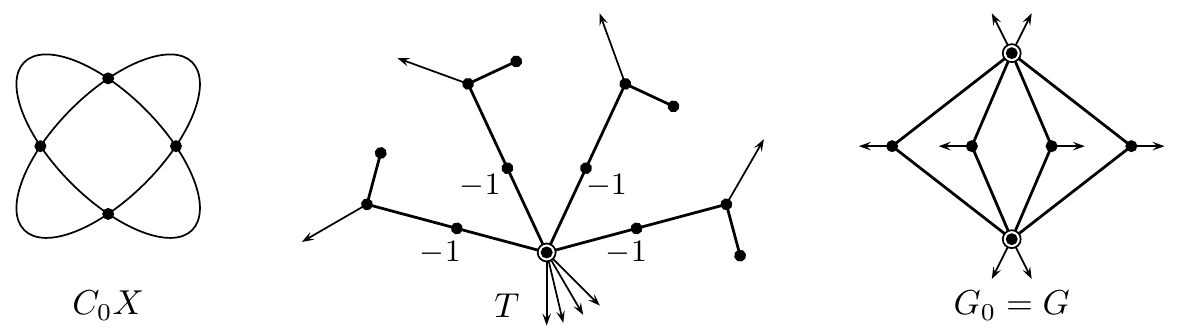}
\caption{Example~\ref{ex:3}: Two ovals with transversal intersections.}
\label{fig:Ex3}
\end{figure}
\end{example}

\begin{example} \label{ex:4} Consider the singularity $(X,0)\subset (\C^3,0)$, for which $C_0X$ consists of four pairwise transversal lines: 
$$ xy(x+y+z)(x-y-z) + z^5 = 0.$$
The graphs $T$ and $G_0$ are represented on Figure \ref{fig:Ex4}.

\begin{figure}[ht]
\includegraphics{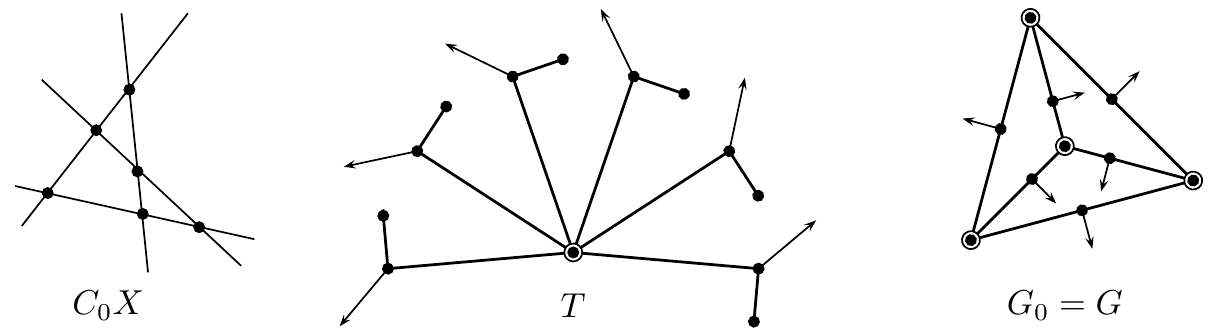}
\caption{Example~\ref{ex:4}: Generic line arrangement.}
\label{fig:Ex4}
\end{figure}
\end{example}

\begin{proof}[Proof of Proposition \ref{prop:Delta}]
Assume that $(X,0)$ has equation $f_d+f_{d+1}=0$. We set $f=f_d$ and $g=-f_{d+1}$. Let $\widetilde{e} \colon W \to \C^3$ be the blow-up of the origin of $\C^3$, so  $X^*$ is the strict transform of $X$ by $\widetilde{e}$  and $e \colon X^* \to X$ is the restriction of $\widetilde{e}$ to $X^*$.

We will work in one of the three standard affine coordinate charts of $W$. Let us fix notations. By definition, $W$ is the closure in $\C^3 \times \mathbb{P}^2$ of the set $$\{ ((x,y,z),L) \in (\C^3\setminus \{0\}) \times \mathbb{P}^2 \colon [x:y:z] = L \},$$ and $\widetilde{e}$ is the  restriction to $W$ of the projection  $\C^3 \times \mathbb{P}^2 \to \C^3$ on the first factor. Consider the affine coordinate chart of $\mathbb{P}^2$ defined by ${U_1 = \{[1:v:w] \in \mathbb{P}^2 \colon (v,w) \in \C^2\}}$ and set  $W_1 = W \cap (\C^3 \times U_1)$. We can choose affine coordinates $(x,v,w)$ in $W_1$ so that  in this chart, $\widetilde{e}(x,v,w) = (x,y,z) = (x,xv,xw)$ and $\widetilde{e}^{-1}(0) \cong \mathbb{P}^2$ has equation $x=0$.

Let $p$ be a point of $C_0X=e^{-1}(0)$ such that the strict transform $\Pi^*_{\cal D}$ passes through $p$. Without loss of generality, we can  choose our coordinates in $\C^3$ so that $p = (0,0,0)$ in the affine coordinates of $W_1$ (i.e., $p$ corresponds to the $x$-axis in the tangent cone of $X$). Then, in this chart, $X^*$ has equation
$$ f(1,v,w) - x g(1,v,w) = 0, $$
and $e^{-1}(0) = C_0X$ has equations $$ f(1,v,w)=0 \quad \hbox{and} \quad x=0.$$
Let us identify  $\grassmann(1, \C^3)$ with $\mathbb{P}^2$. We  can also assume that $\cal D =  [0:0:1]$,  i.e.,  $\ell_{\cal D} = (x,y)$. Then $\Pi_{\cal D}= X \cap \{f_z - g_z = 0\}$, so the strict transform   $\Pi_{\cal D}^*$ of $\Pi_{\cal D}$ by $e$ has equations:
$$f_z(1,v,w) - x g_z(1,v,w) =0 \quad\text{and}\quad f(1,v,w) -x g(1,v,w) =0. $$ 
Since $\Pi^*_{\cal D}$ passes through $p$, we have $f_z(1,0,0) = 0$. 

Set $\widetilde{f}(v,w)=f(1,v,w)$, so $e^{-1}(0)$ has equation $\widetilde{f}(v,w)=0$. Notice that $f_y(1,v,w) = \widetilde{f}_v(v,w)$ and $f_z(1,v,w) = \widetilde{f}_w(v,w)$.
  
Assume first that $p$ is a smooth point of $e^{-1}(0)$, i.e.,  one of $\widetilde{f}_v(0,0)$ or  $\widetilde{f}_w(0,0)$ is nonzero. Since $\widetilde{f}_w(0,0) = f_z(1,0,0) = 0$, we have $\widetilde{f}_v(0,0) = f_y(1,0,0) \neq 0$. Let us prove that $p$ is not a basepoint of the family of generic polars. Since ${\cal D = [0:0:1] \in \Omega}$, for all pairs of sufficiently small complex numbers $(\alpha,\beta)$, we have $\cal D_{(\alpha, \beta)} = [\alpha: \beta: 1] \in \Omega$. The  polar curve $\Pi_{\cal D_{(\alpha, \beta)}}$ has equation $\alpha (f_x-g_x) + \beta (f_y - g_y) +  (f_z - g_z )= 0$, so its strict transform by $\widetilde{e}$ has equations 
$$ \alpha f_x(1,v,w) +  \beta f_y(1,v,w) + f_z(1,v,w)   \hskip3cm $$ 
$$\hskip3cm  - x \big[ \alpha g_x(1,v,w) +  \beta g_y(1,v,w)  + g_z(1,v,w) \big]=0 $$
and
$$ f(1,v,w) -x g(1,v,w) =0.$$ 
Hence a point $(0,v,w) \in e^{-1}(0)$ lies on the strict transform $\Pi^*_{\cal D_{(\alpha,\beta)}}$ if and only if
$$\alpha f_x(1,v,w) +  \beta f_y(1,v,w) + f_z(1,v,w)  =0.$$ 
At $p=(x,v,w)=(0,0,0)$, we have  $\alpha f_x(1,0,0) +  \beta f_y(1,0,0) + f_z(1,0,0) \neq 0$ for generic values of the pair $(\alpha, \beta)$ since $ f_y(1,0,0) \neq 0$. We conclude that  for generic values of $(\alpha,  \beta)$, the strict transform $\Pi^*_{\cal D_{(\alpha,\beta)}}$  does not pass through $p$. This proves that $p$ is not a basepoint of the family of generic polars $(\Pi_{\cal D'})_{\cal D' \in \Omega}$.

Assume now that  $p$ is a singular point of $e^{-1}(0)$,  so $p=p_i$ for some $i=1,\ldots,r$. We then have $ f_x(1,0,0) = f_y(1,0,0) = f_z(1,0,0) =0$ and the previous argument shows that for all $(\alpha, \beta)$,  $\Pi^*_{\cal D_{(\alpha,\beta)}}$  passes through $p$. Therefore $p$ is a basepoint of the family of generic polars $(\Pi_{\cal D'})_{\cal D' \in \Omega}$. This completes the proof of (1).

Let us now prove (2).   For a moment, we just assume that   $(X,0)$ is a superisolated singularity and we prove a lemma.  We use again the above notations. Let $p$ be a singular point of $e^{-1}(0)$, say $p=p_1$. We restrict ourselves to pairs   $(\alpha,\beta)=(0,t)$ for $t \in \C$ and  we set $\cal D_t = \cal D_{(0,t)}$. Since $\cal D_0  = \cal D
\in \Omega$, then  for $t$ in a small neighbourhood $B$ of $0$ in $\C$, we still have $\cal D_t \in \Omega$, i.e. the projection $\ell_t \colon (X,0) \to (\C^2,0)$ defined by $\ell_t (x,y,z) = (x, y-tz)$ is generic for $(X,0)$.

Recall that $\Pi_{\cal D_t,1}$ is the part of the polar curve $\Pi_{t}$ whose strict transform by $\widetilde{e}$ passes through $p=p_1$, i.e., the germ $(\Pi_{{\cal D_t}, 1},0)$ is the $\widetilde{e}$-image of the germ $(\Pi_{\cal D_t}^*,p)$. We will describe the family $(\Pi_{{\cal D_t}, 1})_{t \in B}$.

We use again the previous  coordinates chart and notations. Let $e_0 \colon Y \to \C^2$ be the blow-up of the origin of $\C^2_{(x,y)}$. We consider $e_0$ in the chart $(x,v) \mapsto (x,y)=(x,xv)$, we set $q=(1,0) \in Y$ in this chart and we denote by $\tilde{\ell}_t \colon (X^*,p) \to (Y,q)$ the projection $(x,v,w) \mapsto (x,v-tw)$. So we have the commutative diagram:
$$
\xymatrix{
 (X^*,p)\ar@{->}[r]^e\ar@{->}[d]^{\tilde \ell_t}&(X,0)\ar@{->}[d]^{\ell_t} \\
 (Y,q)\ar@{->}[r]^{e_0}&  (\C^2,0)\hbox to 0pt{\,.\hss} 
}
$$

The series $g(1,v,w)  \in \C\{v,w\}$  is a unit at $p$ since $\{g=0\} \cap Sing (f=0)=\emptyset$ in $\mathbb{P}^2$ and by change of the local coordinates, we can assume $g(1,v,w)=1$. 
Then  the   equations for $(\Pi_{\cal D_t}^*,p)$ can be written:
$$t \widetilde{f}_v(v,w)  +  \widetilde{f}_w(v,w)  =0 \quad\text{and}\quad \widetilde{f}(v,w) - x =0\,.$$

The first equation is nothing but the  polar curve $\Gamma_{t}$ of the morphism  $(v,w) \mapsto (\widetilde{f}(v,w),v-tw)$. Consider the isomorphism $\proj_t \colon (X^*,p) \to (\C^2,0)$ which is the restriction of the linear projection $(x,v,w)\mapsto (v-tw,w)$. We have proved:

\begin{lemma} \label{lem:relative polar} $(\Pi_{\cal D_t}^*,p)$ is the inverse image by $\proj_t$ of the polar curve $\Gamma_{t}$ of the morphism $\phi_t \colon (\C_{(v,w)}^2,0) \to (\C_{(x,v)}^2,0)$ defined by $(v,w) \mapsto (\widetilde{f}(v,w),v-tw)$, i.e., the relative  polar curve of the map germ $  \widetilde{f}$ for the generic projection $(v,w) \mapsto v-tw$.
\end{lemma}

Set $q=(1,0)$ in $\C^2_{(x,v)}$.  The situation is summarized in the  commutative diagram: 
$$
\xymatrix{
(\C^2,\Gamma_t,0)\ar@{<-}[r]^{\proj_t}\ar@{->}[dr]^{\phi_t} & (X^*,\Pi_{\cal D_t}^*,p)\ar@{->}[r]^e\ar@{->}[d]^{\tilde \ell_t}&(X,\Pi_{\cal D_t},0)\ar@{->}[d]^{\ell_t} \\
& (Y,\phi_t(\Gamma_t),q)\ar@{->}[r]^{e_0}&  (\C^2,\Delta_{\cal D_t}, 0) 
}
$$

Assume now that the projectivized tangent cone of the superisolated singularity $(X,0)$ has only ordinary singularities. Then the germ $(e^{-1}(0),p)$ consists of $k$ smooth branches having $k$ distinct tangent lines at $p$.  Maybe after change of coordinates, we can assume  that none of the tangent lines of $(e^{-1}(0),p)$  is tangent to $v=0$. Hence $\widetilde{f}(v,w) = \prod_{i=1}^k (w+a_iv+h.o.),$ where $h.o.$ means higher order terms and where for all $i$,  $a_i \in \C^*$ and $a_i \neq a_j$ for all $i \neq j$.

Assertion~(2a) of Proposition~\ref{prop:Delta} is now a direct consequence of Lemma \ref{lem:relative polar} and of  the following Claim.
\vskip0,3cm \noindent
{\bf Claim 1.} The polar curve $\Gamma_t$ consists of $k-1$ smooth curves which are pairwise transversal. Moreover, the base point $0$ of the family $(\Gamma_t)_{t \in B}$ is resolved by just blowing-up once the origin.
\begin{proof}[Proof of Claim 1.] The curve $\Gamma_t$  has equation $t \widetilde{f}_v(v,w) + \widetilde{f}_w(v,w) = 0$, so $(\Gamma_t)_{t \in \C}$ is a linear pencil of curves which are the fibers of the meromorphic function  $h = \frac{\widetilde{f}_v}{\widetilde{f}_w}$.
We have
$$\widetilde{f}_v(v,w) =  \sum_{i=1}^k a_i  \prod_{j \neq i}(w+a_jv)  \ \  +\  \  \widetilde{\alpha}(v,w)  $$
and
$$ \widetilde{f}_w(v,w) = \sum_{i=1}^k \prod_{j \neq i}(w+a_jv) \ \ + \   \  \widetilde{\beta}(v,w),$$
where $ \widetilde{\alpha}$ and $ \widetilde{\beta}$ are sums of monomials of degree $\geq k$.

Let $e' \colon Y' \to \C^2_{(v,w)}$ be the blow-up of the origin of $ \C^2_{(v,w)}$ and let $C= {e'}^{-1}(0)$ be its exceptional curve. In order to prove that $(\Gamma_t)_{t \in \C}$ is resolved by $e'$, we have to prove that $h \circ e'$ is well  defined along $C$, i.e., that the strict transforms of $\widetilde{f}_v$ and $\widetilde{f}_w$ intersect $C$ at distinct points (see \cite[Sections 2,3]{LW}).  Consider the two polynomials $ \widehat{\alpha}(v,w_1)= \frac{1}{v^k}  \widetilde{\alpha}(v,vw_1)$ and $ \widehat{\beta}(v,w_1)= \frac{1}{v^k} \widetilde{\beta}(v,vw_1)$. In the chart $e' \colon (v,w_1) \mapsto (v,v w_1)$, the strict transforms of $\widetilde{f}_v=0$ and $\widetilde{f}_w=0$ have equations respectively $P(w_1) + v \widehat{\alpha}(v,w_1)=0$ and $Q(w_1) +v \widehat{\beta}(v,w_1) =0$, where 
$$P(w_1) =  \sum_{i=1}^k a_i  \prod_{j \neq i}(w_1+a_j) \ \hbox{ and } \ \ Q(w_1)= \sum_{i=1}^k   \prod_{j \neq i}(w_1+a_j),$$
so their intersections  with $C$ are the points $(v,w_1) = (0,\lambda)$ where the $\lambda$'s are the roots of $P$ (resp. $Q$).

The end of the proof of Claim 1 uses the following:
\vskip0,3cm\noindent
{\bf Claim 2.} The resultant of the polynomials $P$ and $Q$ has the form
$$\Res(P(X), Q(X)) = \eta \prod_{i \neq j}(a_i-a_j)^2,$$
where $\eta \in \C \setminus\{0\}$.
\vskip0,3cm\noindent
 
Since the $a_i$'s are pairwise distinct, $P$ and $Q$ do not have a common root. This proves that  the strict transforms of  $\widetilde{f}_v$ and $\widetilde{f}_w$ intersect $C$ at distinct points. Since $e'$ is a resolution of the meromorphic function $h$, by \cite[Affirmation~1, p.~361]{LW}, $e'$ is a resolution of each generic $\Gamma_t$. Therefore $\Gamma_t$ consists of $k-1$ smooth curves which are pairwise transversal. This proves Claim~1.
\end{proof}
 
\begin{proof}[Proof of Claim 2.] $R:=\Res(P(X), Q(X))$ is a polynomial in the variables $a_1,\ldots,a_k$ which is symmetric since $P$ and $Q$ are themselves symmetric polynomials in these variables. For $a_i-a_j=0$, the monomial $w_1+a_i$ divides both $P(w_1)$ and $Q(w_1)$, so $-a_i$ is a double common root of $P(w_1)$ and $Q(w_1)$, hence $R=0$. This implies that  $a_i-a_j$ divides the polynomial $R$. Since $R$ is symmetric, actually, $(a_i-a_j)^2$ divides $R$. Therefore, the polynomial $S = \prod_{i \neq j}(a_i-a_j)^2$ has to divide $R$. Since $S$ has degree $k(k-1)$, we just have to prove that $R$ has degree at most  $k(k-1)$. We have
$$P = p_1w_1^{k-1} + p_2 w_1^{k-2} + \ldots + p_{k-1} w_1 + p_k$$
and
$$Q = w_1^{k-1} + q_1 w_1^{k-2}+ \ldots +q_{k-2} w_1 +  q_{k-1},$$
where for each $i \in \{1\ldots,k-1\}$, $p_i$ and $q_i$ are homogeneous symmetric polynomials of degree $i$ and where $p_k = ka_1 a_2 \ldots a_k$ is of degree $k$. The resultant $R$ is the determinant of the following $(k-1) \times(k-1)$ matrix: 

$$ \begin{pmatrix}
     p_1 & 0 & \cdots & 0 &1 & 0 & \cdots & 0\\ 
     p_2 & p_1 & \ddots & \vdots & q_1 & 1 & \ddots&0 \\
     \vdots & p_2 & \ddots & 0 & \vdots & q_1 & \ddots & 0 \\
     p_{k-1} & \vdots & \ddots & p_1 & q_{k-2} & \vdots & \ddots & 1 \\[5pt]
     p_k & p_{k-1} & & p_2 & q_{k-1} & q_{k-2} & & q_1 \\
     0 & p_k & \ddots & \vdots& 0 & q_{k-1} & \ddots & \vdots \\
     \vdots & \ddots & \ddots & p_{k-1} & \vdots & \ddots & \ddots & q_{k-2} \\[5pt]
     0 & \cdots & 0 & p_k & 0 & \cdots & 0 & q_{k-1}
\end{pmatrix} $$

\noindent Writing this matrix as  $\begin{pmatrix} a_{ij} \end{pmatrix}$, its determinant $R$ equals  
$$R = \sum_{\sigma \in \cal S_{2k-2}} \hbox{sign}(\sigma)\prod_{i=1}^{2k-2} a_{i \sigma(i)},$$
where $\cal S_{2k-2}$ is the group of permutations on a set of $2k-2$ elements. When $\sigma$ is the identity permutation, the corresponding summand is the homogenous polynomial $(p_1)^{k-1}(q_{k-1})^{k-1}$, whose degree is $k(k-1)$. Now, it is easy to see from the matrix above that performing a transposition of two indices $i$ and $j$ either leads to a zero term or to a homogeneous monomial of the same degree $k(k-1)$. Therefore, $R$ is a sum of  homogeneous monomials of degree $k(k-1)$. This completes the proof of Claim 2. 
\end{proof}

Let us now prove (2b). Since the family of plane curves $(\ell_{\cal D}(\Pi_{\cal D'}))_{\cal D \times \cal D'}$ is equisingular (\cite[Chap.~I, 6.4.2]{teissier}), it suffices to prove it for $\ell = \ell_0 = (x,y)$, $\cal D = \cal D_t$ for any $t$ close to $0$, and for a pair  $(\cal D, \cal D' ) = (\cal D_t, \cal D_{0})$ with $t \neq 0$.

Let $\lambda_{j,t}$, $j=1,\ldots,k-1$, be the $k-1$ roots of  the polynomial $t P + Q$. We denote by $\delta_{j,t}$ the $\widetilde{e}$-image of the component $\delta^*_{j,t}$ of $(\Pi_{\cal D_t}^*,p)$ which has a parametrization of the form $w=\lambda_{j,t} v + h.o.$ Substituting into the equation $\widetilde{f}(v,w) - x =0$, we obtain that $\widetilde{\ell}(\delta^*_{j,t})$ is a smooth curve parametrized by  $x= \mu_{j,t} v^k +h.o.$, where $\mu_{j,t} = \prod_{i=1}^k(\lambda_{j,t}+a_i)$. Therefore  ${\ell}(\delta_{j,t})$ is a $(k+1,k)$-cusp. Moreover, for small~$t$, we have  $\mu_{j,t} \neq \mu_{l,t}$ for every $j \neq l$. This implies that the two cusps ${\ell}(\delta_{j,t})$ and ${\ell}(\delta_{l,t})$ have contact $\frac{k+1}{k}$. So $\ell(\Pi_{{\cal D_t},1})$ is a union of $k-1$ cusps of type $(k+1,k)$ having pairwise contact $\frac{k+1}{k}$.

Moreover, the two sets $\{ \mu_{j,0}\ \colon j = 1,\ldots,k-1\}$ and $\{ \mu_{j,t}\ \colon j = 1,\ldots,k-1\}$ are disjoint for $t \neq 0$, which means that  the components of $ \ell(\Pi_{\cal D_t,1})$ and $\ell(\Pi_{\cal D_{0},1})$ have  pairwise contact $\frac{k+1}{k}$. This completes the proof of (2b). 
\end{proof}

\section{Proof of Theorem \ref{th:main}} 

We first state two lemmas proved in \cite{NPP1} which will enable one to check Conditions $(1^*)$ and $(2^*)$ of Theorem \ref{thm:characterization LNE} in the proof of the ``if" direction of Theorem \ref{th:main}. 

Let $\ell \colon (X,0) \to (\C^2,0)$ be a generic projection, let $\gamma$ be a test curve and let $\widehat{\gamma}$ be a principal component of $\ell^{-1}(\gamma)$. The equality $\mult(\gamma) = \mult(\widehat{\gamma})$ of Condition $(1^*)$ can be checked easily in the resolution as follows.  Let $C$ be the exceptional curve of $\rho_{\ell}^{-1}(0)$ which intersects  the strict transform $\gamma^*$. Let $\pi \colon \widetilde{X} \to X$ be a   resolution of $X$ which is also a good resolution of  $\widehat{\gamma}$, i.e. $\pi^{-1}(0) \cup \widehat{\gamma}^*$ is a normal crossing divisor, and let $E$ be the exceptional component of $\pi^{-1}(0)$ which intersects $\widehat{\gamma}^*$. Let $m_C$ be the multiplicity of a generic linear form $h \colon \C^2 \to \C$  along $C$ and let $m_E$ be that of a generic linear form $k \colon (X,0) \to (\C,0)$ along $(X,0)$.  

\begin{lemma}(\cite[Lemma 11.6]{NPP1}) \label{lem3} Assume that Conditions $(1^*)$ and $(2^*)$  are satisfied for any test curve at the root vertex (so for any generic line through the origin).  Let $\gamma$ be a  test curve at  a component $C$ of $\rho_{\ell}^{-1}$(0) which is not the root vertex.  Then any pair of principal components $\widehat{\gamma}_1,  \widehat{\gamma}_2$ of $\ell^{-1}(\gamma)$ whose strict transforms intersect  the projectivized tangent cone $C_0X=e^{-1}(0)$ at distinct points satisfies Condition~$(2^*)$.
\end{lemma}

We now state Proposition~7.2 of \cite{NPP1} (This statement in~\cite{NPP1} is given in terms of real arcs. Here, we state an equivalent statement in terms of complex curves).  First let us introduce some notations.  Let $(X,0) \subset (\C^n,0)$ be a complex surface and let ${\mathscr N } \colon NX \to X$ be the Nash modification of $X$. The {\it Gauss map} $\widetilde{\lambda} \colon NX \to \grassmann(2,\C^n) $ is  the restriction to $NX$ of the projection of $X  \times \grassmann(2,\C^n)$ on the second factor.

If $C$ is an irreducible component of the exceptional divisor of a composition of blow-ups $\rho \colon Y \to \C^2$, we defined the {\it inner rate} $q_C$ of $C$ as the contact  in $\C^2$ between the $\rho$-image of two curvettes of $C$ meeting $C$ at distinct points.

\begin{lemma}(\cite[Proposition 7.2]{NPP1}) \label{LNE along strings} Let
$\rho \colon Y \to \C^2$ be a sequence of blow-ups of points which resolves the base points of the family of projected polar curves $(\ell(\Pi_{\cal D}))_{\cal D\in \Omega}$. Let $E'$ be the union of components of $ \rho^{-1}(0)$ which are not $\Delta$-curves (Definition~\ref{def:test curve}). Let $(\gamma,0) \subset (\C^2,0)$ be  a complex curve such that $\gamma^* \cap E' \neq \emptyset$  and such that $\gamma^*$ intersects $\rho^{-1}(0)$ at a smooth point.
Let $C$ be the component of $ \rho^{-1}(0)$ such that $C \cap \gamma^* \neq \emptyset$ and let $q_C$ be its inner rate.   Let  $\gamma_1$ and $\gamma_2$ be two  components of $\ell^{-1}(\gamma)$ and consider the two points $p_1 = \gamma_1^*\cap {\mathscr N }^{-1}(0)$ and  $p_2 = \gamma_2^*\cap {\mathscr N }^{-1}(0)$ {where $^*$ means strict transform by the Nash modification $\mathscr N$}. Assume that $q_{inn}(\gamma_1,\gamma_2) = q_C$. Then the pair of curves $(\gamma_1,\gamma_2)$ satisfies Condition~$(2^*)$ if and only if $ \widetilde{\lambda}(p_1) \neq \widetilde{\lambda}(p_2)$.
\end{lemma}

\begin{proof}[Proof of Theorem \ref{th:main}] The ``only if" direction of the theorem is a direct consequence of the main result of~\cite{FS} which states that a LNE analytic set has a LNE tangent cone. In the case of a hypersurface singularity in $(\C^3,0)$, having a LNE tangent cone is equivalent to having a projectivized tangent cone which is LNE. But a complex curve singularity is LNE if and only if it consists of pairwise transversal smooth components. Therefore, if a supersisolated singularity is LNE, its projectivized tangent cone has only  ordinary singularities.

Let us prove the ``if" direction of the theorem.

Assume $(X,0)$ is a superisolated singularity with equation  $f_d + f_{d+1}=0$ whose  projectivized tangent cone  $\{f_d=0\} \subset \mathbb{P}^2$ has only  ordinary singularities. We have to prove that for all generic projections $\ell$, every test curve satisfies Conditions~$(1^*)$ and~$(2^*)$ of Theorem~\ref{thm:characterization LNE}.

Let $\ell \colon \C^3 \to \C^2$ be a generic linear projection for  $(X,0)$, let $\Pi$ be its polar curve  and  let $\Delta = \ell(\Pi)$ be its discriminant curve.

First, consider a test curve at the root vertex of $T$, i.e., $\gamma$ is a generic line in $(\C^2,0)$. Since the projectivized tangent cone $C_0(X) = \{f_d=0\} \subset \mathbb{P}^2$ is reduced,  the inverse image $\ell^{-1}(\gamma)$ consists of $d$ smooth curves  which meet pairwise  transversally. Therefore, $\gamma$ satisfies Conditions~$(1^*)$ and~$(2^*)$.

Let us now consider a  test curve  $\gamma$ at a component $C$ of $\rho_{\ell}^{-1}$(0) which is not the root vertex. We first prove Condition~$(1^*)$.

We will use the notations of Proposition~\ref{prop:Delta} and the description of the trees~$T$ and~$G_0$ given just after its statement. Recall that $\pi_0 \colon X_0 \to X$ denotes the minimal resolution of $(X,0)$ which factorizes through the Nash modification and through the blow-up of the maximal ideal. By Proposition~\ref{prop:Delta}, $\pi_0$ is obtained by first considering the blow-up $e \colon X^* \to X$ of the maximal ideal  of $(X,0)$ and then blowing-up each singular point $p_i$, $i=1,\ldots,r$ of the projectivized tangent cone $C_0X$ on the smooth surface $X^*$.

Assume first that $\gamma$ is a test curve at the $\Delta$-node $v_i$ of the subtree $T_i$. Call~$C_i$ the  component of $\rho_{\ell}^{-1}(0)$ corresponding to $v_i$. Then the multiplicity $m_{C_i}$ of a generic linear form $h \colon (\C^2,0) \to (\C,0)$ along $C_i$ equals $k_i$. On the other hand, the principal components of $\ell^{-1}(\gamma)$ are $k_i$ curvettes of the  component  $E_i$ of $\pi_0^{-1}(0)$ obtained by blowing-up the point $p_i$ on the smooth surface $X^*$.  Since $E_i \subset X_0$ is the exceptional curve of the blowing-up of a point on a smooth surface, its  self-intersection $E_i^2$ in the surface $X_0$ equals~$-1$. Let us compute the multiplicity $m_{E_i}$ of  a generic linear form $H$ on $(X,0)$ along $E_i$. The multiplicity of~$H$ equals~$1$ along each component of $C_0X$. If~$(H)$ denotes the total transform of $H$ by $\pi_0$, we have $(H).E=0$ for each component $E$ of $\pi_0^{-1}(O)$ (\cite[Lemma~2.6]{L}). Since $(H) = \sum_{j=1}^r  m_{E_j} E_j + C_0X$, we obtain for $E=E_i$: $m_{E_i} E_i^2 + E_i.C_0X=0$. This gives $m_{E_i} = k_i$, since ${E_i.C_0X=k_i}$. Therefore $m_{C_i} =m_{E_i}$. Hence we have proved that each of the $k_i$ principal components of  $\rho_{\ell}^{-1}(0)$ satisfies Condition~$(1^*)$.

Assume now that the root vertex of $T$ is a $\Delta$-node and let $\gamma$ be a test curve at  the separation node obtained by blowing-up the intersection point $q_i$ of the exceptional curve $e_0^{-1}(0)$ with  $C_i$. Let $C'_i$ be the curve created by blowing-up $q_i$. Then the multiplicity of $h$ along $C'_i$ equals $k_i+1$. let $\widehat{\gamma}$ be one of the principal connected components of $\ell^{-1}(\gamma)$. Then  $q' = \widehat{\gamma}^* \cap \pi_0^{-1}(0)$ is the intersection point between $E_i$ and $C_0X$, and we obtain a resolution of $\widehat{\gamma}$ by blowing-up the point  $q'$. Let $E'_i$ be the exceptional curve of this blowing-up. Then the multiplicity $m_{E'_i}$ of $H$ along $E'_i$ equals $m_{E_i} + 1 = k_i+1$. We then have $m_{C'_i} = m_{E'_i}$, i.e., $\widehat{\gamma}$ satisfies Condition $(1^*)$.

Let us now prove Condition $(2^*)$. By Lemma~\ref{lem3}, we will then have to check Condition~$(2^*)$ just for pairs of principal components of $\ell^{-1}(\gamma)$ whose strict transforms by $e$ intersect the same singular point of $C_0X=e^{-1}(0)$.

Assume first  that $\gamma$ is a test curve at the $\Delta$-node $v_i$ of $T_i$, i.e., the strict transform $\gamma^*$ of $\gamma$ by $e_0$ is of the form $x=\mu v^k$ in the local coordinates defined above, with $\mu \in \C$ generic. Any point of $C_0X$ distinct from $p_i$ is on at most one component of the strict transform $\gamma^*$ of $\gamma$ by $e_0$, so we just have to consider components of $\gamma^*$ passing through $p_i$. Let $\widehat{\gamma}_i$ be the union of components of  $\ell^{-1}(\gamma)$ whose strict transforms by $e$ pass through $p_i$.  If $\delta$ is one of them, it has a Puiseux expansion of the form $w = \lambda v + h.o. $  in the local coordinates $(x,v,w)$ centered at $p_i$. Substituting this into the equation $\widetilde{f}(v,w)-x=0$ of $X^*$, we obtain $\prod_{i=1}^k (\lambda + a_i)v + h.o. -\mu v^k=0$. Therefore $\lambda$ satisfies the equation $ \prod_{i=1}^k (\lambda + a_i) - \mu =0$.
For a generic value of $\mu$, this equation admits $k_i$ distinct roots. This means that the curve $\widehat{\gamma}_i$ has exactly $k_i$ components and its strict transform by $\pi_0$ consists of $k_i$  curvettes of the $\cal P$-curve $E_i$ meeting $E_i$ at  $k_i$ distinct points.

Let $\delta_{i,1}$ and $\delta_{i,2}$ be two components of  $\widehat{\gamma}_i$ and let  $w = \lambda_1 v + h.o. $ and  $w = \lambda_2 v + h.o. $ be the Puiseux expansions of their strict transforms by $e$. Then the curve germs $\delta_{i,1}$ and $\delta_{i,2}$ of $(\C^3,0)$ are parametrized by 
$$ v  \longmapsto (\mu v^k, \mu v^{k+1}, \lambda_j \mu v^{k+1}), \quad j=1,2.$$
Since $\lambda_1 \neq \lambda_2$, this implies $q_o(\delta_{i,1},\delta_{i,2}) = \frac{k+1}{k}$.

On the other hand, by \cite[15.3]{NPP1}, the inner rate $q_i(\delta_{i,1},\delta_{i,2})$ equals the inner rate between two curvettes of $C_i$ meeting $C_i$ at distinct points. Since $C_i$ is the node of the  resolution tree $T_i$ of the $(k+1,k)$-cusp, we obtain $q_i(\delta_{i,1},\delta_{i,2})= \frac{k+1}{k}$.

Finally, $q_o(\delta_{i,1},\delta_{i,2})=q_i(\delta_{i,1},\delta_{i,2})$, so $\gamma$ satisfies Condition  $(2^*)$.

Assume now that $\gamma$ is a curvette of a curve $C'_i$ corresponding to a separating node and let $\widehat{\gamma}_i$ be as above. The curve $\widehat{\gamma}_i$ consists of $k_i$ curves whose strict transforms by $\pi_0$  pass through the $k_i$ points $p_{i,1}, \ldots, p_{i,k_i}$ of  $C_0X \cap E_i$. By Lemma~\ref{LNE along strings}, in order to prove that all pairs of components of $\widehat{\gamma}_i$ satisfy Condition~$(2^*)$, we have to prove that for all $r \neq  s$, $\widetilde{\lambda}(p_{i,r}) \neq  \widetilde{\lambda}(p_{i,s})$. Let $C_j$ be the component of $C_0X$  whose strict transform by the blow-up of $p_i$ satisfies $C_j \cap E_i = p_{i,j}$. The tangent line $L_{i,j}$ to $C_j$ at $p_i$ corresponds to a $2$-plane in $\grassmann(2,\C^n)$, and we have $\widetilde{\lambda}(p_{i,j}) = L_{i,j}$. Since the singularity $(C_0X,p_i)$ is ordinary, we have $L_{i,r} \neq L_{i,s}$ for all $r \neq s$, i.e., $\widetilde{\lambda}(p_{i,r}) \neq \widetilde{\lambda}(p_{i,s})$. This completes the proof.
\end{proof}

\end{document}